\documentclass[intlimits]
{article}
\usepackage{amsmath}
\usepackage{amsfonts}
\usepackage{amssymb}
\usepackage{amsthm}
\usepackage{enumerate}
\usepackage[latin1]{inputenc}
\usepackage[T1]{fontenc}
\usepackage{t1enc}
\usepackage{parskip}
\usepackage{a4wide}

\def\eq{equation}

\def\tk{\widetilde{K(r)}}

\def\ep{\varepsilon}

\def\S*{\Sigma_*}

\def\r{\rho}
\def\<{\langle}
\def\>{\rangle}

\def\N{\mathbb{N}}

\def\R{\mathbb{R}}
\def\rd{\mathbb{R}^d}

\def\H{\mathcal{H}}
\def\S{\mathcal{S}}

\def\1{\mathbf{1}}
\newcommand{\const}{\operatorname{const}}

\newcommand{\diam}{\operatorname{diam}}

\newcommand{\esup}{\operatorname{ess\,sup}}

\newcommand{\nor}{\operatorname{nor}}
\newcommand{\Int}{\operatorname{int}}
\newcommand{\rea}{\operatorname{reach}}
\newcommand{\Frac}{{\rm frac}}
\newcommand{\var}{{\rm var}}

\swapnumbers
\theoremstyle{plain}
   \newtheorem{thm}{Theorem}[section]
   \newtheorem{thms}{Theorem}[subsection]

   \newtheorem{lems}[thms]{Lemma}
   \newtheorem{lem}[thm]{Lemma}

\theoremstyle{remark}
\theoremstyle{definition}

   \newtheorem{rems}[thms]{Remark}
   
   \newtheorem{expls}[thms]{Examples}
   
   \newtheorem{examp}[thm]{Example}
\begin{document}
\title{Curvature densities of self-similar sets}
\author{J. Rataj\thanks{Supported by grants MSM 0021620839 and GA\v CR 201/10/J039.} \hskip 5pt and M. Z\"ahle \thanks{Supported by grant DFG ZA 242/5-1.} \\[5mm]
      (Charles University Prague and Friedrich Schiller University Jena,\\[1mm] 
       e-mail: rataj@karlin.mff.cuni.cz, martina.zaehle@uni-jena.de)}
\date{}

\maketitle
\begin{abstract}

For a large class of self-similar sets $F$ in $\rd$, analogues of the higher order mean curvatures of differentiable submanifolds are introduced, in particular, the fractal Gauss-type curvature. They are shown to be the densities of associated fractal curvature measures, which are all multiples of the corresponding Hausdorff measures on $F$, due to its self-similarity. This local approach based on ergodic theory for an associated dynamical system enables us to extend former global curvature results.
\end{abstract}
\section{Introduction}
In recent years first attempts to investigate a second order fractal 'differential' geometry have been made by Winter and the second author, see \cite{Wi06}, \cite{Za10}, \cite{WiZa}, \cite{Wi10}. The main idea was to approximate fractal sets in $\rd$ by small neighborhoods and to use known results from singular curvature theory in convex geometry and, more generally, geometric measure theory for these neighborhoods provided they have the desired structure. It turned out that this is the case for many self-similar sets satisfying the open set condition. In order to obtain limit results for the appropriately rescaled global curvatures the renewal theorem from probability theory and asymptotic analysis has been used. (For the special case of the Minkowski content in $\R^1$ this goes back to Lapidus \& Pomerance \cite{LP93} and Falconer \cite{Fa97} and in $\rd$ to Gatzouras \cite{Ga00}.) Moreover, weak limits of the corresponding curvature measures have been obtained as a consequence taking into regard the self-similarity property and Prohorov's theorem on weak compactness of tight families of measures.

In the present paper we suggest another approach.
We start with a result concerning the existence of local fractal curvatures at almost all points of the self-similar set $F$ (Section 2). (The differential geometric analogue are the symmetric functions of principal curvatures of smooth submanifolds.) For this we mainly use the scaling properties of the curvature measures and Birkhoff's ergodic theorem for an associated dynamical system. The positive reach assumption on the closure of the complement of the parallel sets of $F$ for almost all distances is the same as in the former papers, but the integrability condition is essentially weakened. However, here we get only convergence results in the sense of average limits.

By the choice of an appropriate net of locally homogeneous neighborhoods $A_F(x,\ep),\, x\in F,\, \ep<\ep_0$, for the construction of these local curvatures, we can easily derive the existence of related global fractal curvatures which simplifies the proofs and extends the corresponding results from \cite{Wi06}, \cite{Za10} and \cite{WiZa}. The weak convergence of the associated curvature measures then follows as in \cite{Wi06} and \cite{WiZa}. Moreover, the local curvatures can now be interpreted as densities of the fractal limit measures with respect to $D$-dimensional Hausdorff measure on $F$, where $D$ equals the Hausdorff dimension. (See Section 3.)

Finally, in Section 4 we study the examples of the Cantor dust in the plane and the Menger sponge in space which demonstrate some typical phenomena. They do not satisfy the assumptions from the former papers: in particular, the Euler number of the parallel sets $F(\ep)$ of $F$ with distance $\ep$ is unbounded at neighborhoods of certain critical values of $\ep$. Nevertheless, they fit into the approach of the present paper. (See also \cite{Wi10} for further conditions and examples.)

\section{Basic notions}
\subsection{Self-similar sets}
The notion of self-similar sets is well-known from the literature (see Hutchinson \cite{Hu81} for the first general approach and the relationships mentioned below without a reference). We use here the following notations and results.\\
The basic space is a compact set $J\subset\rd$ with $J=\overline{\Int J}$. $S_1,\ldots,S_N$ denotes the generating set of {\it contracting similarities} in $\rd$ with {\it contraction ratios} $r_1,\ldots,r_N$. We assume the {\it strong open set condition} (briefly (SOSC)) with respect to $\Int J$, i.e.,
$$\bigcup_{j=1}^N S_j(J)\subset J\, ,~~  S_j(\Int J)\cap S_l(\Int J)=\emptyset\, ,~j\neq l,$$
 and that there exists a sequence of indices $l_1,l_2,\ldots ,l_m\in \{1,\ldots,N\}$ such that
 $$S_{l_1}\circ S_{l_2}\ldots\circ S_{l_m}(J)\cap \Int J\neq\emptyset\, .$$
 The latter (strong) condition is here equivalent to
 $$F\cap\Int J\ne\emptyset$$
 where $F$ denotes the associated self-similar fractal set $F$. (According to a result of Schief \cite{Sch} (SOSC) for some $J$ is already implied by the open set condition on the similarities. In view of this paper a characterization of (SOSC) in algebraic terms of the $S_i$ is given in Bandt and Graf \cite{BG92}.) The set $F$ may be constructed by means of the {\it code space} $W:=\{1,\ldots,N\}^\N$, the set of all infinite words over the alphabet $\{1,\ldots,N\}$. We write $W_n:=\{1,\ldots,N\}^n$ for the set of all words of length $|w|=n$, $W_*:=\bigcup_{n=1}^\infty W_n$ for the set of all finite words,
  $w|n:=w_1w_2\ldots w_n$ if $w=w_1w_2\ldots w_n,w_{n+1}\ldots$ for the {\it restriction} of a (finite or infinite) word to the first $n$ components, and  $vw$ for the {\it concatenation} of a finite word $v$ and a word $w$. If $w=w_1\ldots w_n\in W_n$ we also use the abbreviations
 $S_w:=S_{w_1}\circ\ldots\circ S_{w_n}$ and $r_w:=r_{w_1}\ldots r_{w_n}$ for the contraction ratio of this mapping. Finally we denote $K_w:=S_w(K)$ for any compact set $K$ and $w\in W_*$. (For completeness we also write $K_\emptyset:=K$.) In these terms the set $F$ is determined by
 $$F=\bigcap_{n=1}^\infty \bigcup_{w\in W_n}J_w$$
 and characterized by the {\it self-similarity property}
 $F=S_1(F)\cup\ldots\cup S_N(F)$.
 Iterated applications yield
$$
F=\bigcup_{w\in W_n}F_w,\quad n\in\N .
$$
As in the literature, we will use the abbreviation
$$S(K):=\bigcup_{j=1}^N S_j(K)$$
for compact sets $K$, i.e., $F=S^n(F)$, $n\in\N$.\\
Alternatively, the self-similar fractal $F$ is the image of the code space $W$ under the {\it projection} $\pi$ given by
$$
\pi(w):=\lim_{n\rightarrow\infty}S_{w|n}x_0
$$
for an arbitrary starting point $x_0$.
The mapping $w\mapsto x=\pi(w)$ is biunique except for a set of points $x$ of $D$-dimensional Hausdorff measure $\H^D$ zero, and the {\it Hausdorff dimension} $D$ of $F$ is determined by
\begin{equation}\label{dim}
\sum_{j=1}^N r_j^D=1\, .
\end{equation}
Up to exceptional points we {\it identify} $x\in F$ with its {\it coding sequence} and write
$x_1x_2\ldots$ for this infinite word, i.e. $\pi(x_1x_2\ldots)=x$, and write $$x|n:=x_1\ldots x_n$$ for the corresponding finite words.\\
If $\nu$ denotes the infinite product measure on $W$ determined by the probability measure on the alphabet $\{1,\ldots,N\}$ with single probabilities
$r_1^D,\ldots,r_N^D$, then the normalized $D$-dimensional Hausdorff measure with support $F$ equals
\begin{equation}\label{bernoulli-meas}
\mu:=\H^D(F)^{-1}\H^D(F\cap(\cdot))=\nu\circ \pi^{-1}\,.
\end{equation}
It is also called the natural {\it self-similar measure} on $F$, since we have
\begin{equation}  \label{ssmu}
\mu=\sum_{j=1}^N r_j^D \mu\circ S_j^{-1}\, .
\end{equation}
Furthermore, by the open set condition $F$ is a $D$-{\it set}, i.e., there exist positive constants $c_F$ and $C_F$ such that
\begin{equation}\label{D-set}
c_F\, r^D\le \H^D(F\cap B(x,r))\le C_F\, r^D\, ,~~x\in F ,~r\le \diam F\, .
\end{equation}
From (SOSC) on $J$ one obtains
\begin{equation}\label{intlogdist}
\int|\ln d(y,J^c)|\mu(dy)<\infty
\end{equation}
(see Graf \cite[proof of Proposition 3.4]{Gr95}). This implies, in particular, that
\begin{equation}\label{vanishing-boundary}
\mu(\partial J)=0
\end{equation}
which can also be seen by other methods.

\subsection{Curvature measures of parallel sets}\label{sec:classcurv}
We will use the following notations for points $x$ and subsets $E$ of $\rd$:
$$d(x,E):=\inf_{y\in E}|x-y|~,~~ |E|:=\diam E=\sup_{x,y\in E}|x-y|\, .$$
The background from classical singular curvature theory is summarized in \cite{Za10}. We recall some of those facts.
For certain classes of compact sets $K\subset\rd$ (including many classical geometric sets) it turns out that for Lebesgue-almost all distances $r>0$ the parallel set
\begin{\eq}\label{parallelset}
K(r):=\{x\in\rd: d(x,K)\le r\}
\end{\eq}
possesses the property that the closure of its complement
\begin{\eq}\label{closcompl}
\widetilde{K(r)}:= \overline{K(r)^c}
\end{\eq}
is a set of positive reach in the sense of Federer \cite{Fe59} with Lipschitz boundary. A sufficient condition is that $r$ is a regular value of the Euclidean distance function to $K$ (see Fu \cite[Theorem 4.1]{Fu85} together with \cite[Proposition 3]{RZ03}). (Recall that in $\mathbb{R}^2$ and $\mathbb{R}^3$ this is fulfilled for all $K$ (see \cite{Fu85}). In this case both the sets $\widetilde{K(r)}$ and $K(r)$ are {\it Lipschitz d-manifolds of bounded curvature} in the sense of \cite{RZ05}, i.e., their {\it k-th Lipschitz-Killing curvature measures}, $k=0,1,\ldots,d-1$, are determined in this general context and agree with the classical versions in the special cases. Moreover, they satisfy
\begin{\eq}  \label{sign}
C_k(K(r),\cdot)=(-1)^{d-1-k}C_k\big(\tk,\cdot\big)\, .
\end{\eq}
Hence, the $C_k(K(r),\cdot)$ are signed measures with finite
{\it variation measures} $C_k^{\var}(K(r),\cdot)$ and the explicit integral representations are reduced to \cite{Za86b} (cf. \cite[Theorem 3]{RZ05} for the general case). In the present paper only the following main properties of the curvature measures for such parallel sets will be used:\\
$C_{d-1}(K(r),\cdot)$ agrees with one half of the $(d-1)$-{\it dimensional Hausdorff measure} $\H^{d-1}$ on the boundary $\partial K(r)$. Note that $\partial K(r)$ is $(d-1)$-rectifiable for any compact set $K$ and any $r>0$ (see \cite[Proposition~2.3]{RW10}), hence, we can always use the notation
$$C_{d-1}(K(\ep),\cdot):=\frac 12\H^{d-1}(K(\ep)\cap(\cdot)).$$
Furthermore, for completeness we define
$C_d(K(r),\cdot)$ as {\it Lebesgue measure restricted to} $K(r)$. The {\it total measures (curvatures)} of $K(r)$ are denoted by
\begin{\eq}
C_k(K(r)):=C_k(K(r),\rd)\, ,~~ k=0,\ldots ,d\, .
\end{\eq}
By an associated Gauss-Bonnet theorem (see \cite[Theorems 2,3]{RZ03}) the {\it total Gauss curvature}
$C_0(K(r))$ coincides with the {\it Euler-Poincaré characteristic} $\chi(K(r))$.\\
The curvature measures are {\it motion invariant}, i.e.,
\begin{\eq}
C_k(g(K(r)),g(\cdot))=C_k(K(r),\cdot)~~\mbox{for any Euclidean motion}~g\, ,
\end{\eq}
they are {\it homogeneous of degree} $k$, i.e.,
\begin{\eq}\label{scaling}
C_k(\lambda K(r),\lambda (\cdot))=\lambda^k\, C_k(K(r),\cdot)\, ,~~\lambda>0\, ,
\end{\eq}
and locally determined, i.e.,
\begin{\eq}\label{loc-curv}
C_k(K(r),(\cdot)\cap G)=C_k(K'(r'),(\cdot)\cap G)
\end{\eq}
for any open set $G\subset\rd$ such that $K(r)\cap G=K'(r')\cap G$, where $K(r)$ and $K'(r')$ are
both parallel sets where the closures of the complements have positive reach.

We shall need the following property of the surface area of parallel sets:
\begin{equation}  \label{sa}
\H^{d-1}(K(r))\leq \frac dr \H^d(K(r)),\quad r>0,
\end{equation}
which follows from the ``Kneser property'' of the volume function, $\frac{{\rm d}}{{\rm d}r}\H^d(K(r))\leq\frac dr\H^d(K(r))$, see e.g.\ \cite{RSS}, Lemma~4.6 and its proof, and from the fact that $\frac{{\rm d}}{{\rm d}r}\H^d(K(r))=\H^{d-1}(K(r))$ up to countably many $r>0$, see \cite{RW10} for more details.

\section{Local curvatures of self-similar sets}
\label{sec:localcurv}
\subsection{Local neighborhood nets}\label{ssec:loc-neighb-nets}
Throughout the paper we will assume the {\it neighborhood regularity} of the self-similar set $F$:
\begin{equation}\label{neighborregularity}
\rea \widetilde{F(\ep)} >0~ \mbox{and}~\partial{F(\ep)}~\mbox{is a Lipschitz manifold for Lebesgue almost all}~\ep>0\, .
\end{equation}
Recall from the previous section that for space dimensions $d\le 3$ this is always fulfilled. It is not difficult to see that it remains true for arbitrary $d$ if the parallel sets $F(\ep)$ are polyconvex.\\
Under the regularity condition for such an $\ep$ the curvature measures $C_k(F(\ep),\cdot)$ are defined. In order to determine some local limits as $\ep\rightarrow 0$ we consider the following notion.

Let constants $a>1$ and $\ep_0>0$ be given and denote $b:=\max\big(2a,\ep_0^{-1}|J|\big)$.
A {\it locally homogeneous neighborhood net in $F$} is a family of sets
$$\{A_F(x,\ep):\, x\in F,\, 0<\ep<\ep_0\}$$
satisfying the following two conditions:
\begin{equation}
A_F(x,\ep)\subset\partial F(\ep)\cap B(x,a\ep),
\end{equation}
\begin{equation}\label{loc-hom-neighbor}
A_F(x,\ep)=S_j\big(A_F(S_j^{-1}x,r_j^{-1}\ep)\big)~~ \mbox{if}~~1\le j\le N,\, x\in F_j~~ \mbox{and}~~ \ep<b^{-1}d(x,(S_jJ)^c)\, .
\end{equation}
(Note that the last inequality implies $r_j^{-1}\ep<\ep_0$.)

\begin{expls}\hfil
\begin{eqnarray}  \label{exp1}
A_F(x,\ep)&:=&\partial F(\ep)\cap B(x,a\ep),\quad \ep>0, \\
\label{exp2}
A_F(x,\ep)&:=&\{z\in \partial F(\ep):\{y\in F:|y-z|=\ep\}\subset B(x,\ep)\}, \quad \ep>0,\\
&&\text{(the set of those points from $\partial F(\ep)$ which have their foot points on $F$} \nonumber\\
&&\text{within the ball $B(x,\ep)$),}\nonumber\\
\label{expl3}
A_F(x,\ep)&:=&\{z\in \partial F(\ep):|x-z|\le\r_F(z,\ep)\}, \quad 0<\ep<\ep_0:=\H^D(F)^{1/D},
\end{eqnarray}
where $\r_F(z,\ep)$ is for $0<\ep<\ep_0$ determined by the condition
$$\rho_F(z,\ep)=\min\{\rho:\, \H^D(F\cap B(z,\rho))=\ep^D\}.$$
\end{expls}

The required properties of the set families \eqref{exp1} and \eqref{exp2} can easily be verified. We shall show that the same is true for \eqref{expl3}. The importance of Example \eqref{expl3} will be clear in Section 3. \\

\begin{lems}
For any $z\in \R^d$, the function $\ep\mapsto\rho_F(z,\ep)$ is well defined and increasing on $(0,\H^D(F)^{1/D})$. The sets \eqref{expl3} define a locally homogeneous neighborhood net in $F$ with parameters $\ep_0=\H^D(F)^{1/D}$ and $a=2c_F^{-1/D}$, where $c_F\leq 1$ is a constant from \eqref{D-set}.
\end{lems}

\begin{proof}
The function $\rho\mapsto\H^D(F\cap B(z,\rho))$ is increasing and takes values between $0$ and $\H^D(F)$. Since it is, moreover, continuous by the Lemma~\ref{spheres}, we can determine $\rho_F(z,\ep)$ from the appropriate level set. Thus, $A_F(z,\ep)$ is well defined for any $x\in F$ and $0<\ep<\ep_0$.

We further infer that for $x\in F$ and $\ep<\ep_0$,
\begin{equation}\label{ball-property}
\rho_F(x,\ep)\leq a\ep\quad\text{and}\quad A_F(x,\ep)\subset \partial F(\ep)\cap B(x,a\ep).
\end{equation}
To see this, note that for any $z\in\partial F(\ep)$ and $y\in F$ such that $|y-z|=\ep$ we have $B(y,\ep)\subset B(z,2\ep)$ and thus the left inequality in \eqref{D-set} applied to $B(y,\ep)$ implies $\H^D(F\cap B(z,2\ep))\ge c_F\ep^D$ and hence, $\r_F(z,c_F^{1/D}\ep)\le 2\ep$, which implies \eqref{ball-property}.

Finally, we shall verify that for $a=2c_F^{-1/D}$ the sets $A_F(x,\ep)$ satisfy condition \eqref{loc-hom-neighbor}. Let $j\in \N$ and $x\in F_j$ be given. Note that if $\ep <b^{-1}d(x,(S_jJ)^c)$ then $B(x,2a\ep)\subset\Int S_jJ$. Then, for any $z\in B(S_j^{-1}x,r_j^{-1}a\ep)$, we have $\rho_F(z,r_j^{-1}\ep)=r_j^{-1}\rho_F(S_jz,\ep)$,
and
$z\in\partial F(r_j^{-1}\ep)$ if and only if $S_jz\in\partial F(\ep)$.
Thus,
\begin{eqnarray*}
A_F(S_j^{-1}x,r_j^{-1}\ep)&=&\{z\in \partial F(r_j^{-1}\ep):|S_j^{-1}x-z|\le\r_F(z,r_j^{-1}\ep)\}\\
&=&\{ z:\, S_jz\in\partial F(\ep), |x-S_jz|\leq\rho_F(S_jz,\ep)\},
\end{eqnarray*}
which implies that $S_j(A_F(S_j^{-1}x,r_j^{-1}\ep))=A_F(x,\ep)$, as required.
\end{proof}

\begin{lems}  \label{spheres}
For any sphere $V$ in $\R^d$ of dimension $k\leq d-1$ we have $\mu(V)=0$.
\end{lems}
\begin{proof}
We shall proceed by induction on $k$. For $k=0$ the assertion is true since $\mu$ is nonatomic. Take a $0<k_0\leq d-1$ and assume that the assertion is true for $k=0,\ldots,k_0-1$. Assume, for the contrary, that $\mu(V)>0$ for some $k_0$-sphere $V$.

Consider now the images $S_iV$, $i=1,\ldots ,N$. These are $k_0$-spheres of lower radii and no two of them coincide (this would contradict the SOSC since we assume $\mu(V)>0$). Consequently, for any $i\neq j$, the intersection $S_i(V)\cap S_j(V)$ is either empty or a sphere of dimension less than $k_0$ and has thus $\mu$-measure zero. It follows from this and from the self-similarity of $\mu$ \eqref{ssmu}
\begin{eqnarray*}
\mu(S_1(V)\cup\cdots\cup S_N(V))
&=&\mu(S_1(V))+\cdots +\mu(S_N(V))\\
&=&r_1^D\mu(V)+\cdots +r_N^D\mu(V)\\
&=&\mu(V).
\end{eqnarray*}
Since all $S_j(V)$ are spheres of radii less than that of $V$, the intersections $S_j(V)\cap V$ must be either empty or spheres of dimensions lower than $k_0$ and, consequently, have $\mu$-measure zero by the induction assumption. We conclude that
$$\mu(F)\ge\mu(V\cup S_1(V)\cup\cdots\cup S_N(V)) =
\mu(V)+\mu(S_1(V)\cup\cdots\cup S_N(V))=2\mu(V).$$
We can continue in the same way: for a natural number $n$, we consider all finite words $w$ of length $|w|\leq n$ and note that, due to (SOSC), for any two finite words $w,w'$, $S_w(V)$ and $S_{w'}(V)$ cannot coincide, unless $w=w'$. (Indeed, take $w\neq w'$ and let $i$ be the least index with $w_i\neq w'_i$. As the sphere $S_{w|i-1}(V)$ has positive $\mu$-measure, in view of \eqref{vanishing-boundary} it must intersect the interior of $J$ and, hence, its images $S_w(V)$ and $S_{w'}(V)$ intersect the disjoint domains $\Int S_{w|i}(J)$ and $\Int S_{w'|i}(J)$, respectively, and cannot coincide.)
Consequently, $\mu(S_w(V)\cap S_{w'}(V))=0$ if $w\neq w'$ by the induction assumption, as in the first part of the proof.
Using now the additivity of $\mu$, we get
$$
\mu(F)\geq\mu\big(\bigcup_{|w|\leq n}S_w(V)\big)
=\sum_{|w|\leq n}\mu(S_w(V))
=\sum_{|w|\leq n}r_w^D\mu(V)
=\big(\sum_{i=0}^n\sum_{|w|=i}r_w^D\big)\mu(V)
=(n+1)\mu(V)
$$
which contradicts $\mu(F)=1$ if $n$ is large enough.
\end{proof}
\subsection{Existence of local curvatures -- formulation of the main result}\label{ssec:exist-curv-dens}

We now can formulate the first {\it main result}.\\

\begin{thms}\label{curvdens}
Let $k\in\{0,1,\ldots,d\}$ and suppose that the self-similar set $F$ in $\rd$ with contraction ratios $r_1,\ldots,r_N$ and Hausdorff dimension $D$ satisfies the strong open set condition w.r.t. $\Int J$. If $k\le d-2$ we additionally suppose the neighborhood regularity \eqref{neighborregularity}. Let $\{A(x,\ep): x\in F,\, \ep<\ep_0\}$, be a locally homogeneous neighborhood net with constants $a>1$ and $\ep_0>0$, and  let $b=\max\big(2a,\ep_0^{-1}|J|)\big)$. Then for $\mathcal{H}^D$-a.a. $x\in F$ the following average limit exists
\begin{equation} \label{dens}
D_{C_k^{\Frac}|F}(x):=\lim_{\delta\rightarrow 0}\frac{1}{|\ln\delta|}\int_\delta^{b^{-1}d(x,J^c)} \ep^{-k}C_k\big(F(\ep), A_F(x,\ep)\big)\, \, \ep^{-1}d\ep
\end{equation}
and equals the constant
\begin{equation}\label{curvdens-int}
\mathcal{H}^D(F)^{-1}\big(\sum_{j=1}^N r_j^D|\ln r_j|\big)^{-1}\int_F\int_{b^{-1}d(y,(S_{y_1}J)^c)}^{b^{-1}d(y,J^c)}\ep^{-k}C_k\big(F(\ep), A_F(y,\ep)\big)\, \, \ep^{-1}d\ep\, \mathcal{H}^D(dy)
\end{equation}
provided the last integral converges absolutely if
$k\le d-2$, and for $k\in\{d-1,d\}$ this is always true.\\
\end{thms}

\begin{rems}
In Section 4 it will be shown that for the choice of $A_F$ as in Example \eqref{expl3} the constant limit values
$D_{C_k^{\Frac}|F}(x)$, defined by \eqref{dens} for $\H^D$-a.a.\ $x\in F$,
may be interpreted as densities of associated fractal curvature measures.
\end{rems}
\begin{rems}
Using (\eqref{intlogdist}, one can see that a sufficient (sharper) condition for the absolute convergence of the integral is
\begin{equation}\label{bounded}
\esup_{\ep<\ep_0,\, y\in F}\limits \ep^{-k} \big|\, C_k\big(F(\ep),A_F(y,\ep)\big)\big|<\infty\, .
\end{equation}
For polyconvex neighborhoods $F(\ep)$ the last property follows, even for the {\it variation measures} $C_k^{\var}(F(\ep),\cdot)$, like in the proof of Lemma 5.3.2 in Winter \cite{Wi06}. (The curvature measures are local, by the open set condition only a bounded number of parallel sets of the smaller copies of $F$ of diameters equivalent to $\ep$ can intersect the set $A_F(x,\ep)$, these sets are unions of bounded numbers of convex sets of diameters equivalent to $\ep$, and the total variation of the $k$-th curvature measure of such a union set is bounded by $\const\ep^k$.) For a more general sufficient condition see Remark~\ref{uniformbound} below and Section~\ref{S-ex}.
\end{rems}
\subsection{An associated dynamical system - proof of the theorem}\label{ssec:dynsyst}
As an essential auxiliary tool for the proof we use the ergodic {\it shift dynamical system} $[W,\nu,\theta]$ on the code space $W$ for the shift operator
$\theta: W\rightarrow W$ with $\theta(w_1w_2\dots):=(w_2w_3\ldots)$. According to \eqref{bernoulli-meas} it induces the ergodic dynamical system $[F,\mu,T]$, where the transformation $T:F\rightarrow F$ is defined for $\mu$-a.a. $x$ by
$$T(x):=(S_j)^{-1}(x)~ \mbox{if}~ x\in S_j(F)\, ,~j=1,\ldots N\, ,$$
taking into regard that $\mu(S_i(F)\cap S_j(F))=0,\,  i\neq j$.
(More general references on this subject may be found, e.g., in Falconer \cite{Fa97}, Mauldin and Urbanski \cite{MU03}.)
In the above identification of a.a. points with their coding sequences we have for such $x$, $$T(x)=\theta(x_1x_2\ldots)\, .$$
Next note that $\ep<b^{-1}d(x,(S_{x|i}J)^c)$ implies $\ep<r_{x|i}^{-1}\ep<\ep_0$, since
$d(x,(S_{x|i}J)^c)=r_{x|i}\, d(T^ix,J^c)$.
From this and $A_F(x,\ep)\subset B(x,a\ep)$ we obtain for Lebesgue-a.a. $\ep$, $\mu$-a.a. $x$, and $i\in\mathbb{N}$ satisfying the first condition the equalities
\begin{eqnarray*}
C_k\big(F(\ep),A_F(x,\ep)\big)&=&C_k\big(F_{x|i}(\ep), A_F(x,\ep)\big)= C_k\big(F_{x|i}(\ep),S_{x|i}\big(A_F(T^ix,(r_{x|i}^{-1}\ep)\big)\big)\\
&=&r_{x|i}^k C_k\big(F(r_{x|i}^{-1}\ep),A_F(T^ix,(r_{x|i}^{-1}\ep)\big)\, .
\end{eqnarray*}
Here we have used the locality \eqref{loc-curv} of the curvature measure $C_k$, the representation \eqref{loc-hom-neighbor} of the sets $A_F(x,\ep)$, and the scaling property \eqref{scaling} of $C_k$ under similarities.\\
Now we will verify the limit
\begin{eqnarray*}
& &\lim_{\delta\rightarrow 0}\limits\frac{1}{|\ln\delta|}\int_\delta^{b^{-1}d(x,J^c)}\limits\ep^{-k}\, C_k\big(F(\ep),A_F(x,\ep)\big)\, \ep^{-1}d\ep\\
&=&\lim_{\delta\rightarrow 0}
\limits\frac{n(x,\delta)}{|\ln\delta|}\frac{1}{n(x,\delta)}\bigg(\sum_{i=0}^{n(x,\delta)-1}\limits\int_{b^{-1}d(x,(S_{x |(i+1)}J)^c)}
^{b^{-1}d(x,(S_{x|i}J)^c)}\ep^{-k}C_k\big(F(\ep),A_F(x,\ep)\big)\, \ep^{-1}d\ep\\
& &\qquad \qquad \qquad \qquad \qquad \quad +\int_\delta^{b^{-1}d(x,(S_{x|n(x,\delta)}J)^c)}\ep^{-k}C_k\big(F(\ep),A_F(x,\ep)\big)\, \ep^{-1}d\ep\bigg) ,
\end{eqnarray*}
where $$n(x,\delta):=\max\{n\in\N: b^{-1}d(x,(S_{x|n}J)^c)\ge\delta\}\, .$$
By the above relationship the integrand in the $i$th integral may be replaced by
$$r_{x|i}^k\, \ep^{-k}\, C_k\big(F(r_{x|i}^{-1}\ep),A(T^ix,r_{x|i}^{-1}\ep)\big)\, \ep^{-1}\, .$$
For the integral bounds we use
\begin{eqnarray*}
d(x,(S_{x|i}J)^c)&=&r_{x|i}\, d(T^ix,J^c) ,\\
d(x,(S_{x|(i+1)}J)^c)&=&r_{x|i}\, d(T^ix,(S_{(T^ix)_1}J)^c)\, .
\end{eqnarray*}
Substituting then under the integral $r_{x|i}^{-1}\, \ep$ by $\ep$ we obtain the expression
$$\int_{b^{-1}d(T^ix,(S_{(T^ix)_1}J)^c)}
^{b^{-1}d(T^ix,J^c)}\ep^{-k}C_k\big(F(\ep),A_F(T^ix,\ep)\big)\, \ep^{-1}d\ep\, .$$
Therefore it suffices to show that for $\mu$-a.a. $x\in F$ the following integrals and limit relationships exist:
\begin{eqnarray}\label{birkhoff}\label{limbirkhoff}
& &\lim_{n\rightarrow\infty}\frac{1}{n}\sum_{i=1}^n\int_{b^{-1}d(T^ix,(S_{(T^ix)_1}J)^c)}
^{b^{-1}d(T^ix,J^c)}\ep^{-k}C_k\big(F(\ep),A_F(T^ix,\ep)\big)\, \ep^{-1}d\ep\\
&=& \int_F\int_{b^{-1}d(y,(S_{y_1}J)^c)}
^{b^{-1}d(y,J^c)}\ep^{-k}C_k\big(F(\ep),A_F(y,\ep)\big)\, \ep^{-1}d\ep\, \mu(dy)\, ,
\end{eqnarray}
\begin{equation}\label{neglible}
\lim_{n\rightarrow\infty}\frac{1}{n}\int_{b^{-1}d(T^nx,(S_{(T^nx)_1}J)^c)}
^{b^{-1}d(T^nx,J^c)}\ep^{-k}\big|C_k\big(F(\ep),A_F(T^nx,\ep)\big)\big|\, \ep^{-1}d\ep=0\,  ,
\end{equation}
(note that under the above conditions $b^{-1}d(T^nx,(S_{(T^nx)_1}J)^c)<\delta$), and
\begin{equation}\label{factor}
\lim_{\delta\rightarrow 0}\frac{|\ln\delta|}{n(x,\delta)}=\sum_{j=1}^Nr_j^D\, |\ln r_j|\, .
\end{equation}
Under the integrability assumption of our theorem \eqref{birkhoff} follows from Birkhoff's ergodic theorem applied to the ergodic dynamical system $[F,\mu,T]$.
(For $k\in \{d-1,d\}$, see Remark~\ref{Rem_int}.)
Here the curvature measures may also be replaced by their absolute values. Taking into regard that $a_n=\sum_{i=1}^na_i-\sum_{i=1}^{n-1}a_i$ for any real sequence, \eqref{neglible} is a consequence.\\
In order to use these arguments for \eqref{factor}, too, note that for $\delta(x,n):=b^{-1}d(x,(S_{x|n}J)^c)$ we get
$$\lim_{\delta\rightarrow 0}\frac{|\ln\delta|}{n(x,\delta)}=\lim_{n\rightarrow\infty}\frac{|\ln\delta(x,n)|}{n}$$
provided the last limit exists. Since
$$\delta(x,n)=r_{x|n}\, b^{-1}d(T^nx,J^c)=\prod_{i=1}^nr_{x_i}\, b^{-1}d(T^nx,J^c)$$
and $x_i=(T^ix)_1\, ,~i\in\N$, Birkhoff's ergodic theorem implies for $\mu$-a.a. $x\in F$
$$\lim_{n\rightarrow\infty}\frac{1}{n}\, |\ln \prod_{i=1}^nr_{x_i}|=\lim_{n\rightarrow\infty}\frac{1}{n}\sum_{i=1}^n|\ln r_{(T^ix)_1}|=\int_F|\ln r_{y_1}|\, \mu (dy)=
\sum_{j=1}^N |\ln r_j|\, r_j^D$$
as well as
$$\lim_{n\rightarrow\infty}\frac{1}{n}\, |\ln d(T^nx,J^c)|=0\, $$
since
 $\int|\ln d(y,J^c)|\, \mu(dy)<\infty$ (cf.\ \eqref{intlogdist}). This shows \eqref{factor} and thus, the proof of the theorem is completed.

\section{Global versus local curvatures}
\subsection{Existence of global curvatures of self-similar sets}
Our final aim is to show under slightly stronger conditions that the local fractal curvatures $D_{C_k^{\Frac}|F}(x)$ from \eqref{dens} may be interpreted as certain densities of associated fractal curvature measures. To this aim we first deduce from Theorem \ref{curvdens} the existence of {\it global fractal curvatures} and establish a relationship to the local versions. (At the same time this provides a simpler proof and a certain extension of the related deterministic result from \cite{Za10} for the global curvatures avoiding the renewal theorem. However, the use of the latter provides more information concerning convergence without averaging over the distances $\ep$.)

\begin{thms}\label{globcurv}
Let $k\in\{0,1,\ldots,d\}$ and suppose that the self-similar set $F$ in $\rd$ with contraction ratios $r_1,\ldots,r_N$ and Hausdorff dimension $D$ satisfies (SOSC) w.r.t.\ $\Int J$. For $k\le d-2$ we additionally assume the neighborhood regularity \eqref{neighborregularity}.
Let $\{A(x,\ep): x\in F,\, \ep <\ep_0=\H^D(F)^{1/D}\}$, be the locally homogeneous neigborhood net of $F$ given in \eqref{expl3} with constant $a=2c_F^{-1/D}$, where $c_F\leq 1$ fulfills \eqref{D-set}, and let
 $b=\max\big(2a,\ep_0^{-1}|J|\big)$. If for $k\le d-2$,
 \begin{equation}  \label{**}
\int_F\sup_{\delta<\ep_0}\frac{1}{|\ln\delta|}\int_\delta^{\ep_0}\ep^{-k}C_k^{\var}\big(F(\ep), A_F(x,\ep)\big)\, \, \ep^{-1}d\ep\, \H^D(dx)<\infty\, ,
\end{equation}
which is always true for $k\in\{d-1,d\}$,
then the following limit exists
$$C_k^{\Frac}(F):=\lim_{\delta\rightarrow 0}\frac{1}{|\ln\delta|}\int_\delta^{\ep_0} \ep^{D-k}C_k\big(F(\ep)\big)\, \, \ep^{-1}d\ep$$
and equals
$$\big(\sum_{j=1}^N r_j^D|\ln r_i|\big)^{-1}\int_F\int_{b^{-1}d(y,(S_{y_1}J)^c)}^{b^{-1}d(y,J^c)}\ep^{-k}C_k\big(F(\ep), A_F(y,\ep)\big)\, \, \ep^{-1}d\ep\, \mathcal{H}^D(dy)\, ,$$
i.e., $$C_k^{\Frac}(F)= D_{C_k^{\Frac}|F}\, \H^D(F)\, .$$
\end{thms}
\vspace{3mm}

\begin{rems}
For $k=d$ this was proved by Gatzouras \cite{Ga00} and for $k=d-1$ by Rataj and Winter \cite[Theorem~4.4]{RW10}. Moreover, it was shown that
$$C_{d-1}^{\Frac}(F)=(d-D)C_d^{\Frac}(F),$$
see \cite[Theorem~4.7]{RW10}.
\end{rems}

\begin{rems}\label{uniformbound}
A sufficient condition for the desired integrability properties is \eqref{bounded} for the variation measures, i.e.,
\begin{equation}  \label{bound*}
\esup_{\ep<\ep_0,\, y\in F}\limits \ep^{-k}C_k^{\var}\big(F(\ep), A_F(y,\ep)\big)<\infty
\end{equation}
which is fulfilled for the special case of polyconvex neighborhoods. More generally, Lemma 3.3 in Winter and Z\"ahle \cite{WiZa} shows that the conditions of Theorem 2.3 in that paper imply the last uniform estimate. Therefore, Theorem \ref{globcurv} extends the first part of that theorem which was proved in \cite[Corollary 2.3.9]{Za10} (cf. the examples in Section 4).
\end{rems}

\begin{proof}
The problem can easily be reduced to Theorem \ref{curvdens} taking into regard the relationship
\begin{equation}\label{int}
\int_F C_k\big(F(\ep), A_F(x,\ep)\big)\, \mathcal{H}^D(dx)
= \ep^D C_k\big(F(\ep)\big)
\end{equation}
for a.a. $\ep<\ep_0$, which follows from our choice of the sets $A(x,\ep)$ according to \eqref{expl3}:
\begin{eqnarray*}
\int_F C_k\big(F(\ep), A_F(x,\ep)\big)\, \mathcal{H}^D(dx)=\int_F\int_{F(\ep)}{\bf 1}\left(|x-z|\le \rho_F(z,\ep)\right)\, C_k(F(\ep),dz)\, \mathcal{H}^D(dx)\\
=\int_{F(\ep)} \H^D\left(F\cap B(z,\rho_F(z,\ep)\right)\, C_k(F(\ep),dz)= \int_{F(\ep)}\ep^D\,  C_k(F(\ep),dz)=\ep^D C_k(F(\ep))\, .
\end{eqnarray*}
Then we get
$$\frac{1}{|\ln\delta|}\int_\delta^{\ep_0} \ep^{D-k}C_k\big(F(\ep)\big)\, \, \ep^{-1}d\ep=\int_F\frac{1}{|\ln\delta|}\int_\delta^{\ep_0} \ep^{-k}C_k\big(F(\ep), A_F(x,\ep)\big)\, \, \ep^{-1}d\ep\, \H^D(dx)\, .$$
The conditions for applying Fubini are guaranteed by the integrability assumption of the theorem.
Moreover, the functions under the last outer integral are uniformly bounded by an integrable function. Therefore we can take the limit as $\delta\rightarrow 0$ under this integral, and Theorem~\ref{curvdens} implies the assertion.
(For the integrability conditions in the case $k\in\{ d-1,d\}$, see Remark~\ref{Rem_int}.)
\end{proof}

\subsection{Fractal curvature measures and their densities}
Because of the self-similarity of $F$ the global curvatures are reflected on its smaller copies. This leads to a method of proving weak convergence of the curvature measures of the parallel sets to fractal limit measures, which was first applied in Winter \cite{Wi06} for the case of polyconvex neighborhoods and was extended in Winter and Z\"ahle \cite{WiZa} to a more general setting. The limit measures $C_k^{\Frac}(F,\cdot)$ were shown to be constant multiples of the normalized Hausdorff measure $\mu$ on $F$, where the constants are equal to the corresponding total fractal curvatures $C_k^{\Frac}(F)$.\\ The following extension provides a local limit interpretation in view of Theorem \ref{curvdens}: The constant local limits $D_{C_k^{\Frac}|F}$ from \eqref{dens}, which are equal to $\H^D(F)^{-1}C_k^{\Frac}(F)$ by Theorem \ref{globcurv}, are the {\it densities of associated fractal curvature measures} $C_k^{\Frac}(F,\cdot )$ with respect to Hausdorff measure $\H^D$ on $F$. Here we need a slightly stronger assumption which is, however, much weaker than the uniform boundedness \eqref{bound*} assumed in the former papers.\\

\begin{thms}\label{curvmeas}
Let $k\in\{0,1\ldots,d\}$ and suppose that the self-similar set $F$ in $\rd$ with contraction ratios $r_1,\ldots,r_N$ and Hausdorff dimension $D$ satisfies (SOSC) w.r.t.\ $\Int J$. For $k\le d-2$ we additionally assume the neighborhood regularity \eqref{neighborregularity}.
Let $\{A(x,\ep): x\in F,\, \ep <\ep_0=\H^D(F)^{1/D}\}$, be the locally homogeneous neigborhood net of $F$ given in \eqref{expl3} and let $a,b>0$ be as in Theorem~\ref{globcurv}. If for $k\le d-2$,
\begin{equation}  \label{*}
\sup_{\delta<\ep_0}\frac{1}{|\ln\delta|}\int_\delta^{\ep_0}\ep^{-k}\sup_{x\in F}\limits C_k^{\var}\big(F(\ep), B(x,a\ep)\big)\, \, \ep^{-1}d\ep<\infty\, ,
\end{equation}
which is always true for $k\in\{d-1,d\}$,
then we get
$$C_k^{\Frac}(F,\cdot ):=\lim_{\delta\rightarrow 0}\frac{1}{|\ln\delta|}\int_\delta^{\ep_0} \ep^{D-k}C_k\big(F(\ep),\cdot \big)\, \, \ep^{-1}d\ep= D_{C_k^{\Frac}|F}\, \H^D(F\cap(\cdot ))$$
in the sense of weak convergence of signed measures, where the density $D_{C_k^{\Frac}|F}$ can be calculated by \eqref{curvdens-int}.
\end{thms}
\vspace{3mm}
\begin{rems}  \label{Rem_int}
Recall that under our assumptions $A(x,\ep)\subset B(x,a\ep)$ for any $x\in F$ and $\ep<\ep_0$ (cf. \eqref{ball-property}). Therefore the integrability condition \eqref{*} implies \eqref{**}.
Moreover, \eqref{**} ensures the the existence of the integrals assumed in Theorem~\ref{curvdens}.
In case $k\in\{ d-1,d\}$, even the stronger condition \eqref{bound*} for general locally homogeneous neighborhood nets $A(F(x,\ep)$ is always satisfied. For $k=d$ it is obvious, and for $k=d-1$ it can be seen as follows:
\begin{eqnarray*}
C_{d-1}(F(\ep),A_F(x,\ep))&=& \frac 12 \H^{d-1}(\partial F(\ep)\cap A_F(x,\ep))
\leq \frac 12 \H^{d-1}(\partial F(\ep)\cap B(y,a\ep))\\
&\leq& \frac 12 \H^{d-1}(\partial [(F\cap B(y,2a\ep))(\ep)])
\stackrel{\eqref{sa}}{\leq} \frac {d}{2\ep} \H^{d}((F\cap B(y,2a\ep))(\ep))\\
&\leq& \frac {d}{2\ep} \H^{d}(B(y,(2a+1)\ep))
=\frac d{2}\omega_d(2a+1)^d\ep^{d-1},
\end{eqnarray*}
where $\omega_d$ denotes the volume of a unit ball in $\R^d$.
\end{rems}

\begin{proof}
Under a stronger boundedness condition this theorem has been shown in \cite{WiZa}. An essential tool is the corresponding global result (in our case Theorem \ref{globcurv}) in combination with Prohorov's theorem on weak compactness of tight families of measures and the invariance properties of $F$ and the measures under consideration. An analysis of the technically involved proof shows that it remains valid in the essential steps under the following changes: Use in the definition \cite[(2,2)]{WiZa} for the set $\Sigma(\ep)$ of finite words $w$ (with length $|w|$) the modified condition  $$r_w|J|<\ep\le r_{w||w|-1}|J|\, .$$
Then the statement
$$C_k^{\var}\big(F(\ep),C\big)\le \const N(C,\ep)\ep^k$$
of \cite[Lemma 3.3]{WiZa} is replaced by the averaged version
\begin{equation}\label{number}
\sup_{\delta<\ep_0}\frac{1}{|\ln\delta|}\int_\delta^{\ep_0}\big(N(C,\ep)\big)^{-1}\ep^{-k}C_k^{\var}\big(F(\ep),C\big)\, \, \ep^{-1}d\ep\, <\infty\, ,
\end{equation}
provided $C$ is a closed subset of $\rd$ and the number of elements $w$ of the set $\Omega(C,\ep)\subset \Sigma(\ep)$ such that the set $F_w(\ep)$ intersects $C$ is bounded by $N(C,\ep)>0$. All implied estimates used in the proof have now to be understood in this average sense. The rest is the same as in \cite{WiZa}.\\
In order to see \eqref{number}, denote $$K:=\sup_{\delta<\ep_0}\frac{1}{|\ln\delta|}\int_\delta^{\ep_0}\ep^{-k}\sup_{x\in F}\limits C_k^{\var}\big(F(\ep), B(x,a\ep)\big)\, \, \ep^{-1}d\ep\, ,$$
which is finite by assumption, and estimate as follows
\begin{eqnarray*}
C_k\big(F(\ep),C\big)&=&C_k\left(F(\ep),C\cap\bigcup_{w\in \Sigma(\ep)}F_w(\ep)\right)\le C_k\left(F(\ep),\bigcup_{w\in \Omega(C,\ep)}F_w(\ep)\right)\\&\le& \sum_{w\in \Omega(C,\ep)}C_k\big(F(\ep),F_w(\ep)\big)\le N(C,\ep)\sup_{x\in F}\limits C_k^{\var}\big(F(\ep), B(x,a\ep)\big)\, ,
\end{eqnarray*}
since for any $w\in\Sigma(\ep)$ the set $F_w(\ep)$ is contained in a ball with midpoint in $F$ and radius $a\ep$.  This shows that the constant $K$ is an upper bound in \eqref{number}.
\end{proof}

\section{Examples}  \label{S-ex}
In this section we provide two examples of self-similar sets fulfilling the integrability assumption \eqref{*}, but violating \eqref{bounded} (and, hence, also \eqref{bound*}). Therefore the methods of Winter and Zähle \cite{WiZa} are not applicable in these cases, whereas those of the current paper are.

We shall use the fact that the positive part of the curvature measure of order $0$ of a parallel set to any compact subset of the plane is bounded by a constant depending on the parallel radius and diameter of the set only.
Hence, it is enough to control the global curvature.

\begin{lem} \label{L*}
Let $F\subset \R^2$ be compact and let $\ep>0$ be a regular value of $F$. Then
$$C_0^+(F(\ep),\R^2)\leq \frac{2}{\ep^2}\left(\frac{|F|}2+\ep\right)^2.$$
\end{lem}

\begin{proof}
Let $\ep$ be a regular value of $F$ in the sense of \eqref{neighborregularity}. Since $\widetilde{F(\ep})$ has positive reach, its generalized principal curvature $k(x,n)$ is defined $\H^{d-1}$-almost everywhere on the unit normal bundle $\nor\widetilde{F(\ep)}$ of $\widetilde{F(\ep})$, see \cite{Za86b}.  Since $F(\ep)$ is an $\ep$-parallel set, it is not difficult to see that the generalized curvatures are bounded from below by $-\ep^{-1}$ (the curvature of the closure of the complement of an $\ep$-disc) whenever they exist. Thus, using the integral representation from \cite{Za86b}, we get the bound
$$C_0^-(\widetilde{F(\ep)})\leq\frac 1{\pi}\int_{\nor\widetilde{F(\ep)}} \frac{\ep^{-1}}{\sqrt{1+k(x,n)^2}}\,\H^1(d(x,n))=\frac 1{\pi\ep}\H^1(\partial F(\ep))$$
(we have used the co-area formula for the projection $(x,n)\mapsto x$ from $\nor\widetilde{F(\ep)}$ to $\partial F(\ep)$ in the last step).
Recalling \eqref{sign}, we infer
$$C_0^+(F(\ep))=C_0^-(\widetilde{F(\ep)})\leq \frac 1{\pi\ep}\H^1(\partial F(\ep)).$$
We further use \eqref{sa} and the isodiametric inequality, and the proof is finished.
\end{proof}

\begin{examp}  \label{CD}
The {\it Cantor dust with similarity factor} $\rho<\frac 12$ is the self-similar set $F\subset\R^2$ given by four similarities
\begin{align*}
S_1(x,y)&=(\rho x,\rho y),\\
S_2(x,y)&=(\rho x+1-\rho,\rho y),\\
S_3(x,y)&=(\rho x,\rho y+1-\rho),\\
S_4(x,y)&=(\rho x+1-\rho,\rho y+1-\rho).
\end{align*}
We shall show that \eqref{*} holds for $F$, whereas \eqref{bounded} does not.
\end{examp}

\begin{proof}
Set $\tau:=\frac 12-\rho$.
If $\ep>\sqrt{2}\tau$ then $F(\ep)$ is contractible and we have $\chi(F(\ep))=1$.
If $\sqrt{1+\rho^2}\tau<\ep<\sqrt{2}\tau$ then $F(\ep)$ is connected with one hole in the middle and we have $\chi(F(\ep))=0$. More generally, for $k=0,1,2,\ldots$, if $\sqrt{1+\rho^{2k+2}}\tau<\ep<\sqrt{1+\rho^{2k}}\tau$ then $F(\ep)$ is still connected, but contains
$$1+4+4\cdot 2+\cdots +4\cdot 2^{k-1}=1+4(2^k-1)$$
holes, hence
$$\chi(F(\ep))=-2^{k+2}+4.$$
For any $k\geq 0$ and $a>1$, since $\ep>\tau$, we have by Lemma~\ref{L*}
\begin{eqnarray*}
C_0^{\var}(F(\ep),B(x,a\ep))&\leq&C_0^{\var}(F(\ep),\R^2)\\
&\leq&|\chi(F(\ep))|+C_0^+(F(\ep),\R^2)\\
&\leq&2^{k+2}+K/\tau^2
\end{eqnarray*}
with some constant $K$ independent of $x$, $k$ and $\ep$.
Similarly, we get for $\sqrt{2}\tau<\ep<\ep_0$
$$C_0^{\var}(F(\ep),B(x,a\ep))\leq 1+K/\tau^2.$$

From the above considerations, it turns out that the critical values of the distance function are
$$\rho^l\tau, \rho^l\tau\sqrt{1+\rho^{2k}},\quad l=0,1,\ldots,\, k=0,1,\ldots,$$
and they are countably many, hence, \eqref{neighborregularity} is fulfilled. It is also clear, however, that $C_0(F(\ep))$ is unbounded at any neighborhhood of the values $\ep=\rho^l\tau$. This implies, in particular, that \eqref{bounded} is not satisfied (see also the discussion in \cite{Wi10}). Indeed, if $\esup_{\ep<\ep_0,y\in F}|C_0(F(\ep),A_F(y,\ep))|\leq Q$ for some constant $Q$ then we would get as in the proof of Theorem~\ref{curvmeas} $\ep^D|C_0(F(\ep))|\leq \H^D(F) Q$, which would contradict the behavior of the Euler-Poincar\'e characteristic of $F(\ep)$ near the citical points, described above.

Assume now that $\rho^l\tau<\ep<\rho^{l-1}\tau$ for some $l\geq 1$. Then $F(\ep)$ consists of $4^l$ disjoint components $(S_\omega F)(\ep)$, where $\omega$ are words of length $l$. For any such word $\omega$, $(S_\omega F)(\ep)$ is a contraction of $F(\rho^{-l}\ep)$.
Given $x\in F$, let $\Sigma_l(x,a\ep)$ denote the set of all words $w\in\Sigma_l$ such that $S_wF$ hits $B(x,a\ep)$. As $\ep<\rho^{l-1}\tau$, it is easy to see that $\Sigma_l(x,a\ep)$ has at most $(a+1)^2$ elements, by the construction of $F$. Thus, if
$$\sqrt{1+\rho^{2k+2}}\rho^l\tau<\ep<\sqrt{1+\rho^{2k}}\rho^l\tau$$
for some $k\geq 0$ then, by the previous case, we have
\begin{eqnarray*}
C_0^{\var}(F(\ep),B(x,2\ep))&\leq&C_0^{\var}\left(\left(\bigcup\nolimits_{w\in\Sigma_l(x,a\ep)}S_w F\right)(\ep),B(x,a\ep)\right)\\
&\leq& (a+1)^2C_0^{\var}\left( (S_wF)(\ep),\R^2\right)\\
\leq&(a+1)^2(2^{k+2}+K/\tau^2).
\end{eqnarray*}
If $\sqrt{2}\rho^l\tau<\ep<\rho^{l-1}\tau$ then, similarly,
$$C_0^{\var}(F(\ep),A_F(y,\ep))\leq (a+1)^2(1+K^2).$$
Having still $l\geq 1$ fixed, we can estimate the integral
\begin{eqnarray*}
\lefteqn{(a+1)^2\int_{\rho^l\tau}^{\rho^{l-1}\tau}C_0^{\var}(F(\ep),B(x,a\ep))\frac{d\ep}{\ep}}\\
&\leq&\int_{\rho^l\tau}^{\rho^{l-1}\tau}\frac{K}{\tau^2}\frac{d\ep}{\ep}+\left(\sum_{k=0}^\infty \int_{\rho^l\tau\sqrt{1+\rho^{2k+2}}}^{\rho^l\tau\sqrt{1+\rho^{2k}}}2^{k+2}\,\frac{d\ep}{\ep} +\int_{\rho^l\tau\sqrt{2}}^{\rho^{l-1}\tau}\frac{d\ep}{\ep}\right)\\
&=&\frac{K|\ln\rho|}{\tau^2}+\sum_{k=0}^{\infty}2^{k+2}\left( \ln \rho^l\tau\sqrt{1+\rho^{2k}} -\ln \rho^l\tau\sqrt{1+\rho^{2k+2}}\right)\\
&&\hspace{3cm}+\left( \ln\rho^{l-1}\tau-\ln\sqrt{2}\rho^l\tau\right)\\
&=&\frac{K|\ln\rho|}{\tau^2}+\sum_{k=0}^{\infty}2^{k+2}\frac 12\left( \ln (1+\rho^{2k}) -\ln (1+\rho^{2k+2})\right)
-\ln\sqrt{2}\rho\\
&\leq&\frac{K|\ln\rho|}{\tau^2}+2\sum_{k=0}^{\infty}2^k(\rho^{2k}-\rho^{2k+2})-\ln(\sqrt{2}\rho)\\
&=&\frac{K|\ln\rho|}{\tau^2}+2\frac{1-\rho^2}{1-2\rho^2}-\ln(\sqrt{2}\rho).
\end{eqnarray*}
Thus, the integral in the first line is bounded by a constant (say $L$) independent of $l$. Similarly one can show that the integral
$$L_0:=\int_{\tau}^{\ep_0}C_0^{\var}(F(\ep),B(x,a\ep))\frac{d\ep}{\ep}$$
is bounded.

To finish the proof, note that if
$$l>l(\delta):=\frac{|\ln\delta|}{|\ln\rho|}+1$$
then $\rho^{l-1}\tau<\delta$. Consequently, we can estimate the integral (uniformly in $x\in F$)
\begin{eqnarray*}
\lefteqn{\int_{\delta}^{\ep_0}C_0^{\var}(F(\ep),B(x,a\ep))\, \frac{d\ep}{\ep}}\\
&\leq&\sum_{1\leq l\leq l(\delta)} \int_{\rho^{l}\tau}^{\rho^{l-1}\tau}C_0^{\var}(F(\ep),B(x,a\ep))\, \frac{d\ep}{\ep}
+\int_{\sqrt{2}\tau}^{\ep_0}C_0^{\var}(F(\ep),B(x,a\ep))\, \frac{d\ep}{\ep}\\
&\leq&l(\delta)\cdot L + L_0.
\end{eqnarray*}
The last expression is of order $O(|\ln\delta|)$, as $l(\delta)$ is, and, thus, \eqref{*} holds.
\end{proof}

The second example will be the Menger sponge. In $\R^3$, we cannot use a bound analogous to Lemma~\ref{L*} (the local principal curvatures will be bounded again, but the curvature measures are given as integrals of products of more than one curvature and we loose the control over the sign). Instead, we shall use the following lemma given bounds for variations of curvature measures under reach and diameter restrictions. We formulate it in general dimension $d$, though we need it here for $d=3$ only.

\begin{lem} \label{L**}
Let $s>0$ be fixed. Then, there exists a constant $\eta$ such that $$C^{\var}_k(K,A)\leq \eta (s+\ep)^d\ep^{k-d}$$
whenever $k\in\{ 0,1,\ldots ,d-1\}$, $\ep>0$, $K$ is a compact subset of $\R^d$ with $\rea K\geq \ep$ and the set $A$ fulfills $|A|\leq 2s$.
\end{lem}

\begin{proof}
We use the local Steiner formula (see, e.g., \cite{Za86b}):
$$\sum_{i=1}^d \omega_ir^iC_{d-i}(K,A)
= \H^d((K(\ep)\setminus K)\cap \Pi_K^{-1}(A)),\quad 0<r<\ep,$$
where $\Pi_K$ is the metric projection onto $K$ and $\omega_i$ is the volume of the unit ball in $\R^i$.
Since the set on the right hand side has diameter less than $2(s+\ep)$, its volume is between $0$ and $\omega_d(s+\ep)^d$ by the isodiametric inequality. Denote
$a_i:=\omega_i\ep^iC_{d-i}(K,A)$, $i=1,\ldots,d$.
Then we have
\begin{equation} \label{LP}
0\leq a_1t+a_2t^2+\cdots +a_dt^d\leq \omega_d(s+\ep)^d,\quad 0\leq t\leq 1.
\end{equation}
This is an infinite system of linear inequalities and it sufficies for us to consider only $d$ of them, say $t=1,\frac 12,\dots ,\frac 1d$. \eqref{LP} then takes the form
$$Ma\in [0,L]^d,$$
where $a=(a_1,\ldots ,a_d)$,
$$M=\left( \begin{array}{cccc}
           1&1&\cdots&1\\
           2^{-1}& 2^{-2}&\cdots& 2^{-d}\\
           \multicolumn{4}{c}\dotfill\\
           d^{-1}& d^{-2}&\cdots& d^{-d}
           \end{array}
\right)$$
and $L=\omega_d(s+\ep)^d$. Since $M$ is regular, we can transform the condition to
$$a\in M^{-1} [0,L]^d=L\cdot M^{-1}[0,1]^d,$$
which already implies the linear in $L$ bounds for $|a_i|$ and, hence, also for $\ep^i|C_{d-i}(K,A)|$, $i=1,\ldots ,d$. The assertion follows easily.
\end{proof}

\begin{examp}
The {\it Menger sponge} is the self-similar set $F$ in $\R^3$ with the similarities
$$S_{ijk}(x,y,z):=\left(\frac {x+i}3,\frac{y+j}3,\frac{z+k}3\right),\quad i,j,k=0,1,2,$$
where at most one of the triple of indices $i,j,k$ may be $1$. (Thus, the number of similarities is $20$.) The Hausdorff dimension is $D=\ln 20/\ln 3$. The fractal $F$ is connected and it fulfills the (SOSC) with $J=[0,1]^3$. A natural construction of $F$ starts with the unit cube $[0,1]^3$ and removes subsequenly $7$, ${20}\cdot 7$, ${20}^2\cdot 7,\ldots$ cubes of edge length $1/3,1/3^2,1/3^3\ldots$. We shall call the ${20}^{l-1}\cdot 7$ removed cubes of edge length $1/3^l$ {\it the removed cubes of $l$th generation}.

We claim again that the $C_0(F(\ep))$ is not locally bounded, but \eqref{*} holds for $k=0$, $1$ or $2$.
\end{examp}

\begin{proof}
Let $l\in\N$ be given and assume that $3^{-l}/2<\ep<3^{-l+1}/2$. Then, $F(\ep)$ is connected with topologically cylindrical holes through the removed cubes of generations $1,\ldots ,l-1$, and, eventually, further topologically spherical holes in the removed cubes of generation $l$. Let $F_0(\ep)$ be the set $F(\ep)$ with all topologically spherical holes filled up. An important observation is that $\rea\widetilde{F_0(\ep)}\geq\ep$. (Indeed, observe that the closest to $\widetilde{F_0(\ep)}$ points that do not have unique footpoints in $\widetilde{F_0(\ep)}$ lie on the edges of removed cubes which have distance $\ep$ from $\widetilde{F_0(\ep)}$.)

If $(3^{-l}/2)\sqrt{2}<\ep<3^{-l+1}/2$ then $F(\ep)=F_0(\ep)$ (there is no topologically spherical hole). Further, let $j\in\{ 0,1,\ldots\}$ and assume that
$$\frac 1{2\cdot 3^l}\sqrt{1+3^{-2j-2}}<\ep< \frac 1{2\cdot 3^l}\sqrt{1+3^{-2j}};$$
then, in each of the  $l$th generation removed cube, there are
$$1+2+2^2+\cdots +2^{j-1}=2^j-1$$
topologically spherical holes in $F^\ep$. Each of the topologically spherical holes (with boundary) has reach $\geq\ep$ and diameter less than $\ep$.

Fix some $a>1$, take an $x\in F$ and let $H_i$, $i=1,\ldots ,p$, be the (closed) topologically spherical holes in $\widetilde{F(\ep)}$ hit by $B(x,a\ep)$. Since $\ep<3^{-l+1}/2$, $B(x,a\ep)$ hits at most $(2a+1)^3$ removed cubes of $l$th generation, hence, $p\leq (2a+1)^3(2^j-1)$. Thus, for given $\ep$, we get, using Lemma~\ref{L**},
\begin{eqnarray*}
C_k^{\var}(F(\ep), B(x,a\ep))&=&C_k^{\var}(\widetilde{F(\ep)},B(x,a\ep))\\
&\leq& C_k^{\var}(\widetilde{F_0(\ep)},B(x,a\ep))+\sum_{i=1}^p C_k^{\var}(H_i)\\
&\leq& \eta \ep^{k-3}(a\ep+\ep)^3+(2a+1)^3(2^j-1)\eta\ep^{k-3}(2\ep)^3\\
&\leq&(A+B\cdot 2^j)\eta\ep^k
\end{eqnarray*}
with some constants $A,B$.
Consequently, we can estimate the integral
\begin{eqnarray*}
\int_{3^{-l}/2}^{3^{-l+1}/2} \ep^{-k}C_k^{\var}(F(\ep),B(x,a\ep))\, \frac{d\ep}{\ep}
&\leq&\sum_{j=0}^\infty \int_{\frac 1{2\cdot 3^l}\sqrt{1+3^{-2j-2}}}^{\frac 1{2\cdot 3^l}\sqrt{1+3^{-2j}}} (A+B\cdot 2^j)\eta\, \frac{d\ep}{\ep}+
\int_{(3^{-l}/2)\sqrt{1+3^{-2}}}^{3^{-l+1}/2} A\eta \,\frac{d\ep}{\ep}\\
&=&\eta\sum_{j=0}^\infty (A+B\cdot 2^j)\frac 12 \left( \ln(1+3^{-2j})-\ln(1+3^{-2j-2})\right)\\
&&+A\eta \left( \ln 3-\ln \sqrt{1+3^{-2}}\right)\\
&\leq& \eta\sum_{j=0}^\infty (A+B\cdot 2^j)\frac 12 (3^{-2j}-3^{-2j-2})+2\ln 3-\frac 12\ln {10}
\end{eqnarray*}
and it is not difficult to see that the last expression is bounded; let $L$ denote its value.
The integral
$$L_0:=\int_{1/6}^{\ep_0}\ep^{-k} C_k^{\var}(F(\ep),B(x,2\ep))\, \frac{d\ep}{\ep}$$
is, of course, bounded as well. The rest of the proof continues similarly as that of Example~\ref{CD}. We set
$$l(\delta):=\frac{\ln 3-\ln 2-\ln\delta}{\ln 3}+1$$
and note that $3^{-l+1}/2<\delta$ whenever $l>l(\delta)$. Thus,
\begin{eqnarray*}
\lefteqn{\int_{\delta}^{\ep_0}\ep^{-k}C_k^{\var}(F(\ep),B(x,a\ep))\, \frac{d\ep}{\ep}}\\
&\leq&\sum_{1\leq l\leq l(\delta)} \int_{3^{-l}/2}^{3^{-l+1}/2}\ep^{-k}C_k^{\var}(F(\ep),B(x,a\ep))\, \frac{d\ep}{\ep}
+\int_{1/6}^{\ep_0}\ep^{-k}C_k^{\var}(F(\ep),B(x,a\ep))\, \frac{d\ep}{\ep}\\
&\leq&l(\delta)\cdot L + L_0,
\end{eqnarray*}
which is of order $O(|\ln\delta|)$ as $\delta\to 0$, and, thus, \eqref{*} holds.

As in the first example, the set of critical values of the distance function is countable in this case (these are the end points of the integration domains used). It is not difficult to see that $C_0(F(\ep))$ is unbounded at any neighborhhood of the values $\ep=3^{-l}/2$ since new topologically spherical holes appear at each value $3^{-l}/2\sqrt{1+3^{-2k}}$, $k\in\N$, increasing the Euler characteristic. It follows as in Example~\ref{CD} that \eqref{bounded} does not hold.
\end{proof}

\end{document}